\documentclass[11pt]{amsart}

\usepackage{graphicx}%
\usepackage{multirow}%
\usepackage{amsmath,amssymb,amsfonts}%
\usepackage{amsthm}%
\usepackage{mathrsfs}%
\usepackage[title]{appendix}%
\usepackage{xcolor}%
\usepackage{textcomp}%
\usepackage{manyfoot}%
\usepackage{booktabs}%
\usepackage{algorithm}%
\usepackage{algorithmicx}%
\usepackage{algpseudocode}%
\usepackage{listings}%
\usepackage{lineno}
\usepackage[colorlinks=true, linkcolor=red, urlcolor=blue, citecolor=blue]{hyperref}
\usepackage{xcolor}
\usepackage{tikz-cd}

\theoremstyle{definition}
\newtheorem{theorem}{Theorem}[section]
\newtheorem{proposition}[theorem]{Proposition}
\newtheorem{lemma}[theorem]{Lemma}
\newtheorem{corollary}[theorem]{Corollary}
\theoremstyle{definition}

\newtheorem{example}[theorem]{Example}
\theoremstyle{remark}
\newtheorem{remark}[theorem]{Remark}

\allowdisplaybreaks[1]
\usepackage[left=2.8cm, right=2.8cm, top=3cm, bottom=4cm]{geometry}

\title[Spectral Properties of Convolution Operators on Non-Separable Sequence Spaces]{Spectral Properties of Convolution Operators on Non-Separable Sequence Spaces}

\author{Sayan Saha}
\address{Department of Mathematics, Indian Institute of Technology Bhilai, Durg, 491002, Chhattisgarh, India}
\email{sayans@iitbhilai.ac.in}

\author{Arnab Patra} 
\address{Department of Mathematics, Indian Institute of Technology Bhilai, Durg, 491002, Chhattisgarh, India}
\email{arnabp@iitbhilai.ac.in}

\subjclass[2020]{Primary 47A10, 47B37, 46B26, 47B01}
\keywords{Convolution operator, Spectrum, Fine spectrum, Non-separable space.}

\begin{document}

\begin{abstract}
This paper provides a complete characterization of the spectrum, point spectrum, continuous spectrum, and residual spectrum of bounded convolution operators acting on the non-separable space $\sigma_\infty$, the dual of the Hahn sequence space. Because standard duality and separability-based techniques fail in this setting, a linear functional approach is employed with the fine spectral decomposition of the unilateral shift operator. Furthermore, as a major application of this methodology, we establish a complete description of the fine spectrum for operator polynomials acting on general sequence spaces without requiring separability. This result provides a unified framework that subsumes classical spectral determinations for general banded triangular matrices across various sequence spaces and enables the explicit computation of fine spectrum on non-separable spaces.
\end{abstract}

	\maketitle

\section{Introduction}\label{sec1}
Let $\mathbb{N}_{0} = \mathbb{N} \cup \{0\}$. The convolution of two complex sequences $a=(a_{n})_{n=0}^\infty$ and $b=(b_{n})_{n=0}^\infty$ in $\mathbb{C}^{\mathbb{N}_{0}}$ corresponds to the Cauchy product of their formal power series $f_{a}(z) = \sum_{n=0}^{\infty}a_{n}z^n$ and $f_{b}(z)=\sum_{n=0}^{\infty}b_{n}z^n$. The resulting convolution sequence $a*b$ has $n$-th term:
\[
(a*b)_{n} = \sum_{k=0}^{n}a_{n-k}b_{k}, \quad n \in \mathbb{N}_0.
\]
For a fixed sequence $a$, the associated convolution operator $T_{a}: \mathbb{C}^{\mathbb{N}_{0}} \to \mathbb{C}^{\mathbb{N}_{0}}$ is defined by $T_{a}(x) := a*x.$ With respect to the cannonical basis of $\mathbb{C}^{\mathbb{N}_{0}}$, $T_a$ can be uniquely represented as an infinite lower triangular Toeplitz matrix:
\[
T_a = \begin{pmatrix}
a_0 & 0 & 0 & 0 & \dots \\
a_1 & a_0 & 0 & 0 & \dots \\
a_2 & a_1 & a_0 & 0 & \dots \\
a_3 & a_2 & a_1 & a_0 & \dots \\
\vdots & \vdots & \vdots & \vdots & \ddots
\end{pmatrix}.
\]\label{MR}
When $a=(a_0, a_1, \dots, a_n, 0, 0, \dots)$ has finite support, $T_a$ reduces to a lower triangular $(n+1)$-banded matrix operator.

The spectral theory of convolution operators and lower triangular banded matrices on classical sequence spaces has been a vibrant area of research. For two-banded operators $B(r,s)$ generated by $a=(r,s,0,0, \dots)$, fine spectral classifications were established on $c_0$ and $c$ by Altay and Ba\c{s}ar \cite{altay2005fine}, on $\ell^1$ and $bv$ by Furkan et al. \cite{furkan2006fine}, and on $\ell^p$ and $bv_{p}$ ($1<p< \infty$) by Bilgi\c{c} and Furkan \cite{bilgicc2008fine}. Subsequent generalizations addressed triple-banded operators $B(r,s,t)$ on $c_0, c$ \cite{furkan2007fine}, and $\ell^p, bv_p$ $(1<p<\infty)$ \cite{furkan2010fine}. More general $(n+1)$-banded operators generated by finite sequences $a=(a_{0}, a_{1}, \dots , a_{n}, 0 ,0, \dots )$ were analyzed on $c, c_{0}$ by Altun \cite{altun2011fine}, and on $c_{0}, c, \ell^{p}$, and $bv_{p}$ ($1 \leq p < \infty$) by Birbonshi and Srivastava \cite{birbonshi2017some}. More recently, the boundedness and spectral properties of convolution operators have been explored on Ces\`aro sequence spaces \cite{ricker2019convolution}, dual Ces\`aro sequence spaces \cite{curbera2023convolution}, and Fr\'echet sequence spaces \cite{ricker2019convolutionF}. However, fine spectral characterizations of such operators acting on non-separable sequence spaces, such as $\ell^{\infty}$, remain largly unexplored in the literature primarily due to unavailability of the duality-based functional analytic tools.

Recall that for a bounded linear operator $T \in \mathscr{B}(X)$ on a complex Banach space $X$, the resolvent set is $\rho(T,X) := \{\lambda \in \mathbb{C} : T-\lambda I \text{ is bijective with bounded inverse}\}$, and the spectrum is $\sigma(T,X) := \mathbb{C} \setminus \rho(T,X)$. Let $\mathscr{N}(T)$ and $\mathscr{R}(T)$ denote the null space and range space of $T$. The spectrum $\sigma(T,X)$ is partitioned into three mutually disjoint subsets:
\begin{enumerate}
    \item[(i)] Point spectrum $\sigma_{p}(T,X)$: the set of $\lambda \in \mathbb{C}$ such that $T-\lambda I$ is not injective ($\mathscr{N}(T-\lambda I) \neq \{0\}$);
    \item[(ii)] Continuous spectrum $\sigma_{c}(T,X)$: the set of $\lambda \in \mathbb{C}$ such that $T-\lambda I$ is injective, $\mathscr{R}(T-\lambda I)$ is dense in $X$, but $T-\lambda I$ is not surjective;
    \item[(iii)] Residual spectrum $\sigma_{r}(T,X)$: the set of $\lambda \in \mathbb{C}$ such that $T-\lambda I$ is injective, but $\mathscr{R}(T-\lambda I)$ is not dense in $X$.
\end{enumerate}
The decomposition of the spectrum into these three parts is referred to as the fine spectrum of the operator.

The main purpose of this paper is to investigate the fine spectrum of bounded convolution operators on the non-separable sequence space $\sigma_{\infty}$, defined as the dual of the Hahn space $h$:
\[ 
\sigma_{\infty} := \Bigl\{x \in \mathbb{C}^{\mathbb{N}_{0}} : \|x\|_{\sigma_{\infty}} = \sup_{n\geq 0}\frac{1}{n+1} \Big|\sum_{k=0}^{n}x_{k}\Big| < \infty \Bigr\}.
\]

We construct an algebraic framework based on linear functional representations and the fine spectral decomposition of the fundamental unilateral shift operator on $\sigma_\infty$. Furthermore, as an independent structural result, we establish that for any arbitrary bounded linear operator $T$ acting on complex Banach sequence spaces, whose fine spectrum is known, the complete fine spectrum of the polynomial operator $p(T)$ can be explicitly derived. Applying this general polynomial framework yields a unified setting that subsumes prior studies on banded triangular matrices across classical sequence spaces, while simultaneously extending these fine spectral subdivisions to non-separable spaces such as $\ell^{\infty}$ and $\sigma_\infty$. Suitable concrete examples for convolution operators and banded lower triangular matrices are provided alongside figures to illustrate the underlying spectral geometry.

The rest of the paper is organized as follows. In Section~\ref{sec2}, we establish key structural and topological properties of the dual Hahn sequence space $\sigma_\infty$. In Section~\ref{sec3}, we derive the complete fine spectral decomposition of the unilateral shift operator on $\sigma_\infty$. Section~\ref{sec4} leverages the shift operator analysis to determine the fine spectrum of discrete convolution operators $T_a$ on $\sigma_\infty$. Finally, Section~\ref{sec5} establishes the general theory for operator polynomials $p(T)$ acting on arbitrary Banach sequence spaces (including non-separable spaces), applies it to triangular banded matrices.
	
	

\section{The Sequence Space \texorpdfstring{$\sigma_\infty$}{sigma infinity}} \label{sec2}

The Hahn sequence space $h$, introduced by Hahn \cite{hahn1922folgen}, is defined as
\[
h:= \Biggl\{x\in \mathbb{C}^{\mathbb{N}_{0}} : \sum_{n=0}^{\infty} (n+1)|x_{n}-x_{n+1}|<\infty \quad \text{and} \quad \lim_{n\to\infty}x_n=0\Biggr\}.
\]
Hahn proved that $h$ is a Banach space under the norm
\[
\|x\|_{h}= \sum_{n=0}^{\infty} (n+1)|x_{n}-x_{n+1}| + \sup_{n\ge 0}|x_{n}|,
\]
and that its $\beta$-dual is isomorphic to the sequence space
\[
\sigma_{\infty}:= \Biggl\{ x \in \mathbb{C}^{\mathbb{N}_{0}} : \|x\|_{\sigma_{\infty}}= \sup_{n\ge 0}\frac{1}{n+1}\Big|\sum_{k=0}^{n}x_{k}\Big|< \infty \Biggr\}.
\]
Goes and Goes \cite[Theorem 3.7]{goes1970sequences} later established that the continuous dual $h'$ is also isomorphic to $\sigma_{\infty}$. Equipped with the norm $\|\cdot\|_{\sigma_{\infty}}$, the space $\sigma_{\infty}$ is a non-separable Banach space \cite[Lemma 3.3]{saha2026convolution}.

For any $x = (x_k)_{k=0}^{\infty} \in \mathbb{C}^{\mathbb{N}_0}$, we denote its sequence of partial sums by $S_n(x) := \sum_{k=0}^n x_k$ for each $n \in \mathbb{N}_0$, so that $\|x\|_{\sigma_\infty} = \sup_{n \ge 0} \frac{|S_n(x)|}{n+1}$. The following key inequality bounds the individual sequence entries in terms of the $\sigma_\infty$-norm.

\begin{lemma}\label{MxSI}
Let $x=(x_{n})_{n=0}^\infty \in \sigma_{\infty}$. Then for each $n\in \mathbb{N}_{0}$,
\[
|x_{n}| \leq (2n+1)\|x\|_{\sigma_{\infty}}.
\]
\end{lemma}
\begin{proof}
By definition of $\|\cdot\|_{\sigma_{\infty}}$, we have $|S_k(x)| \le (k+1)\|x\|_{\sigma_{\infty}}$ for all $k \in \mathbb{N}_0$. For $n=0$, \[|x_0| = |S_0(x)| \le \|x\|_{\sigma_\infty}. \] 
For any $n \ge 1$, writing $x_n = S_n(x) - S_{n-1}(x)$ yields
\[
|x_{n}| \le |S_n(x)| + |S_{n-1}(x)| \le (n+1)\|x\|_{\sigma_{\infty}} + n\|x\|_{\sigma_{\infty}} = (2n+1)\|x\|_{\sigma_{\infty}},
\]
which completes the proof.
\end{proof}

Before examining the structural properties of $\sigma_\infty$, we recall some standard sequence space terminologies. A sequence space $X$ endowed with a linear topology is called a $K$-space if every coordinate projection map $p_{n}:X \to \mathbb{C}$, $p_{n}(x)=x_{n}$, is continuous for all $n\in \mathbb{N}_{0}$. A complete metric $K$-space is called an $FK$-space, and a normable $FK$-space is called a $BK$-space. Let $c_{00}$ denote the set of all finitely supported complex sequences. Given a sequence space $X \supseteq c_{00}$, the $n$-th section of $x=(x_{k})_{k=0}^{\infty} \in X$ is $x^{[n]}=\sum_{k=0}^{n}x_{k}e_{k}$, where $e_{k}=(\delta_{k,j})_{j=0}^{\infty}$ and $\delta_{k,j}$ denotes the Kronecker delta. The space $X$ is said to possess the:
\begin{itemize}
    \item $AK$ property if $\lim_{n \to \infty}\|x-x^{[n]}\|_{X}=0$ for all $x \in X$;
    \item $AB$ property if $\sup_{n \geq 0}\|x^{[n]}\|_{X} < \infty$ for all $x \in X$;
    \item $AD$ property if $c_{00}$ is dense in $X$ (i.e., $\overline{c_{00}} = X$).
\end{itemize}
Furthermore, $X$ is called $\beta$-perfect if $X^{\beta\beta}=X$, where $X^\beta$ is the $\beta$-dual of $X$; and $X$ is solid if and only if $\tilde{X} \subseteq X$ where 
\[
        \Tilde{X}:=\Bigl\{ (u_{n})_{n=0}^{\infty}\in \mathbb{C}^{\mathbb{N}_{0}}: \textrm{There exists} \ (x_{n})_{n=0}^{\infty} \in X \ \textrm{such that}\ |u_{n}| \leq |x_{n}| \ \textrm{for all} \ n \in \mathbb{N}_{0} \Bigr\}. 
    \]
For more detail refer to \cite{grosse2001l1,wilansky2000summability}.\\

The main topological and structural properties of $\sigma_\infty$ are summarized in the following proposition.

\begin{proposition}\label{prop:sigma_properties}
The sequence space $\sigma_\infty$ satisfies the following properties:
\begin{enumerate}
    \item[\rm (i)] It is non-separable.
    \item[\rm (ii)] It is a $BK$-space, and hence an $FK$-space and $K$-space.
    \item[\rm (iii)] It satisfies the $AB$ property, but fails both the $AK$ and $AD$ properties.
    \item[\rm (iv)] It is a $\beta$-perfect sequence space.  
    \item[\rm (v)] It is not solid.
\end{enumerate}
\end{proposition}
\begin{proof}
\begin{enumerate}
    \item[\rm (i)] This follows from \cite[Lemma 3.3]{saha2026convolution} with weight sequence $d_{n}=n+1$.
    
    \item[\rm (ii)] By Lemma~\ref{MxSI}, $|p_n(x)| = |x_n| \le (2n+1)\|x\|_{\sigma_\infty}$ for every $n \in \mathbb{N}_0$, which establishes that each coordinate projection map $p_n$ is a bounded linear functional on $\sigma_\infty$. Since $\sigma_\infty$ is a Banach space with continuous coordinate projections, it is a $BK$-space (see also \cite[Remark 2.2]{malkowsky2021compact}).

    \item[\rm (iii)] To verify the $AB$ property, take any $x \in \sigma_\infty$ and $n \in \mathbb{N}_0$. For $m \le n$, $S_m(x^{[n]}) = S_m(x)$, while for $m > n$, $S_m(x^{[n]}) = S_n(x)$. Hence, for all $m \ge 0$,

    \begin{align*}
    \frac{|S_m(x^{[n]})|}{m+1} &\le \begin{cases}
        \frac{|S_m(x)|}{m+1}, \quad \mathrm{for}\ m\le n; \\
        \frac{|S_n(x)|}{m+1}, \quad  \mathrm{for}\ m > n;
    \end{cases}\\ 
     &\le \begin{cases}
        \frac{|S_m(x)|}{m+1}, \quad \mathrm{for}\ m\le n; \\
        \frac{n+1}{m+1}\frac{|S_n(x)|}{n+1}, \quad  \mathrm{for}\ m > n;
    \end{cases}\\
    &\le \|x\|_{\sigma_\infty}.
    \end{align*}
    Taking the supremum over $m \ge 0$ yields $\|x^{[n]}\|_{\sigma_\infty} \le \|x\|_{\sigma_\infty}$, so $\sup_{n \ge 0} \|x^{[n]}\|_{\sigma_\infty} \le \|x\|_{\sigma_\infty} < \infty$, proving $AB$.

    To see that $AK$ fails, consider $x \in \mathbb{C}^{\mathbb{N}_{0}}$ defined by
    \[
    x_{n} = \begin{cases}
    -n, \quad \text{when}\ n \ \text{is even};\\
    n+1, \quad \text{when}\ n \ \text{is odd}.
    \end{cases}
    \]
    Direct calculation gives $S_{2k}(x) = 0$ and $S_{2k+1}(x) = 2k+2$, so \[\frac{|S_n(x)|}{n+1} = \begin{cases}
        0, \quad \text{when}\ n \ \text{is even};\\
        1, \quad \text{when} \ n \ \text{is odd};
    \end{cases}\] implying $\|x\|_{\sigma_{\infty}} = 1$ and $x \in \sigma_{\infty}$. However, evaluating the norm of $x - x^{[2n]}$ yields
    \[ \big\| x - x^{[2n]} \big\|_{\sigma_{\infty}} = \big\| \big(  0, 0, \ldots, 0, (2n+2), -(2n+2), (2n+4), \ldots \big)  \big\|_{\sigma_{\infty}} = 1. \] 
    Therefore $\big\| x - x^{[2n]} \big\|_{\sigma_{\infty}} \not\to 0$ as $n \to \infty$ and hence $\sigma_{\infty}$ does not possess the $AK$ property.\\

    To show that $\sigma_{\infty}$ does not have the $AD$-property, consider another element $x' \in \mathbb{C}^{\mathbb{N}_{0}}$ such that
            \[
            x'=(1,-2,3,\dots, (-1)^n(n+1), \dots).
            \]
            Then, $\|x\|_{\sigma_{\infty}}=1$. Now, if possible let, for $\epsilon= \frac{1}{8}$, there exists a $y \in c_{00}$, such that, $\|x' -y\|_{\sigma_{\infty}}< \frac{1}{8}$. As $y\in c_{00}$, there exists $m\in \mathbb{N}$ such that $y_{k}=0$ for all $k\geq m$. Now, $\|x' -y\|_{\sigma_{\infty}}< \frac{1}{8}$ implies 
            \[
            \frac{1}{n+1}\Big|\sum_{k=0}^{n}(x_{k}'-y_{k})\Big| < \frac{1}{8} \ \textrm{for all}\ n\in \mathbb{N}_{0}.
            \]
            Choosing $n=m$ we have
            \begin{align*}
                &\frac{1}{m+1}\Big|\sum_{k=0}^{m}(x_{k}'-y_{k})\Big| < \frac{1}{8}\\
                \implies & 1-\frac{m}{m+1}\frac{1}{m}\Big|\sum_{k=0}^{m-1}(x_{k}'-y_{k})\Big| < \frac{1}{8}\\
                \implies & 1-\frac{1}{8} < \frac{m}{m+1}\frac{1}{m}\Big|\sum_{k=0}^{m-1}(x_{k}'-y_{k})\Big| < \frac{m}{m+1} \frac{1}{8}\\
                \implies &\frac{7}{8} < \frac{m}{8(m+1)},
            \end{align*}
            which is not possible. Hence, $\sigma_{\infty}$ does not have the $AD$-property.

    \item[\rm (iv)] Since $h^{\beta} = \sigma_{\infty}$, the result follows directly from \cite[Theorem 3.5]{saha2026convolution} by setting $d_{n}=n+1$.

    \item[\rm (v)] Consider $x = \bigl((-1)^{n}(n+1)\bigr)_{n=0}^\infty$ and $u = (n+1)_{n=0}^\infty$. Clearly $|u_{k}| \le |x_{k}| = k+1$ for all $k \in \mathbb{N}_{0}$. As shown in part (iii), $x \in \sigma_{\infty}$ with $\|x\|_{\sigma_\infty} = 1$. However, $S_n(u) = \frac{(n+1)(n+2)}{2}$, so $\frac{S_n(u)}{n+1} = \frac{n+2}{2} \to \infty$ as $n \to \infty$, meaning $u \notin \sigma_{\infty}$. Thus $\sigma_{\infty}$ is not solid.
\end{enumerate}
\end{proof}

\section{Fine Spectrum of the Unilateral Shift Operator}\label{sec3}

In this section, we establish the complete fine spectral decomposition of the unilateral (right) shift operator $R$ on the non-separable sequence space $\sigma_{\infty}$. For any sequence $x=(x_{0},x_{1},x_{2},\ldots)\in \sigma_{\infty}$, the right shift operator $R:\sigma_{\infty}\to\sigma_{\infty}$ is defined by
\[
R(x):=\bigl(0,x_{0},x_{1},x_{2},\dots\bigr).
\]
The operator $R$ plays a fundamental role in our analysis, serving as the elementary building block for general convolution operators $T_a$. We begin by establishing the boundedness and spectrum of $R$, followed by its fine spectral subdivisions.

	\begin{proposition}
		The right shift operator $R$ is a bounded linear operator on $\sigma_{\infty}$ with norm 1, i.e., $\big\| R \big\|_{\mathscr{B}(\sigma_{\infty})} = 1$.
	\end{proposition}
	\begin{proof}
		The linearity of $R$ is obvious. To prove the boundedness of $R$, let $x = \big(x_{0}, x_{1}, x_{2}, \ldots\big)$ be an arbitrary element of $\sigma_{\infty}$. Then
		\begin{align*}
			\big\| R(x) \big\|_{\sigma_{\infty}} 
			&= \big\| \big( 0,x_{0}, x_{1}, x_{2}, \cdots\big) \big\|_{\sigma_{\infty}} \\
			&= \underset{n \geq 1}{\sup}\frac{1}{n+1} \Bigg|\sum_{k=0}^{n-1}x_{k}\Bigg| \\
			&= \underset{n \geq 1}{\sup}\frac{n}{n+1} \frac{1}{n} \Bigg|\sum_{k=0}^{n-1}x_{k}\Bigg| \\
			& \leq \underset{n \geq 1}{\sup} \frac{n}{n+1} \ \  \underset{n \geq 1}{\sup}  \frac{1}{n} \Bigg|\sum_{k=0}^{n-1}x_{k}\Bigg| \\
			&= \|x\|_{\sigma_{\infty}}.
		\end{align*}
		Therefore, 	$\big\| R \big\|_{\mathscr{B}(\sigma_{\infty})} \leq 1$.	On the other hand $x= \big(1, 1, 1, \ldots\big) \in \sigma_{\infty}$ with $\|x\|_{\sigma_{\infty}}=1.$ Then $\big\|R(x)\big\|_{\sigma_{\infty}} = \underset{n}{\sup}\ \frac{n}{n+1} = 1$.
	\end{proof}

\begin{proposition}
The spectrum of the right shift operator $R$ on $\sigma_{\infty}$ is the closed unit disk, i.e.,
\[
\sigma(R, \sigma_{\infty}) = \overline{\mathbb{D}}.
\]
\end{proposition}
\begin{proof}
Since $\|R\|_{\mathscr{B}(\sigma_{\infty})} = 1$, it follows immediately that $\sigma(R, \sigma_{\infty}) \subseteq \overline{\mathbb{D}}$.

Now let $\lambda \in \mathbb{D}$. We show that $R - \lambda I$ is not surjective on $\sigma_{\infty}$. Consider the sequence $y = (-1, 0, 0, \dots) \in \sigma_{\infty}$. Suppose there exists a sequence $x = (x_n)_{n=0}^{\infty} \in \sigma_{\infty}$ such that $(R - \lambda I)x = y$. This yields the system:
\[
-\lambda x_0 = -1 \quad \text{and} \quad x_{n-1} - \lambda x_n = 0 \quad \text{for all } n \ge 1.
\]
Clearly $\lambda \neq 0$. Solving the recurrence relation gives $x_n = \frac{1}{\lambda^{n+1}}$ for all $n \in \mathbb{N}_0$. Evaluating the norm of $x$ gives
\begin{align*}
    \|x\|_{\sigma_{\infty}} 
    &= \sup_{n \ge 0} \frac{1}{n+1} \left| \sum_{k=0}^{n} x_k \right| \\
    &= \sup_{n \ge 0} \frac{1}{n+1} \left| \sum_{k=0}^{n} \frac{1}{\lambda^{k+1}} \right| \\
    &= \frac{1}{|1 - \lambda|} \sup_{n \ge 0} \frac{1}{n+1} \left| \left( \frac{1}{\lambda} \right)^{n+1} - 1 \right|.
\end{align*}
Since $|\lambda| < 1$, we have $|\lambda|^{-1} > 1$, which implies $\|x\|_{\sigma_{\infty}} = \infty$, so $x \notin \sigma_{\infty}$.

Thus, $R - \lambda I$ is not surjective for any $\lambda \in \mathbb{D}$, which proves that $\mathbb{D} \subseteq \sigma(R, \sigma_{\infty})$. Since the spectrum of a bounded linear operator is closed, we obtain $\overline{\mathbb{D}} \subseteq \sigma(R, \sigma_{\infty}) \subseteq \overline{\mathbb{D}}$, establishing $\sigma(R, \sigma_{\infty}) = \overline{\mathbb{D}}$.
\end{proof}

Having identified the spectrum of the right shift operator, we next investigate its fine spectral decomposition. Since $(R-\lambda I)$ is injective for every $\lambda\in\mathbb{C}$, it follows immediately that
\[
\sigma_p(R,\sigma_\infty)=\emptyset.
\]
The following result shows that every spectral value of $R$ belongs to the residual spectrum; consequently, the continuous spectrum of $R$ is empty.

    \begin{theorem}\label{R_residual}
		The residual spectrum of the right shift operator $R$ on $\sigma_{\infty}$ is given by \[\sigma_{r}(R,\sigma_{\infty}) = \overline{\mathbb{D}}.\]
	\end{theorem}

\begin{proof}
Let $\lambda \in \overline{\mathbb{D}}$. The injectivity of $R - \lambda I$ implies that $\mathscr{N}(R - \lambda I) = \{0\}.$ The rest of the proof is divided into three cases covering the closed unit disk $\overline{\mathbb{D}}$.

\textbf{Case 1: $|\lambda| < 1$ ($\lambda \in \mathbb{D}$).} \\
For each $\lambda \in \mathbb{D}$, define a linear functional $f_{\lambda} : \sigma_{\infty} \to \mathbb{C}$ by
\[
f_{\lambda}(x) = \sum_{n=0}^{\infty} \lambda^n x_{n}, \quad x = (x_n) \in \sigma_\infty.
\]
By Lemma~\ref{MxSI}, every $x \in \sigma_{\infty}$ satisfies $|x_n| \le (2n+1)\|x\|_{\sigma_{\infty}}$ for all $n \in \mathbb{N}_0$. Consequently,
\[
|f_{\lambda}(x)| \le \sum_{n=0}^{\infty} |\lambda|^n |x_{n}| \le \|x\|_{\sigma_{\infty}} \sum_{n=0}^{\infty} (2n+1)|\lambda|^n.
\]
Since $\sum_{n=0}^{\infty} (2n+1)|\lambda|^n < \infty$ for $|\lambda| < 1$, $f_{\lambda}$ is a bounded linear functional on $\sigma_{\infty}$ with 
\[ \|f_\lambda\| \le \sum_{n=0}^{\infty} (2n+1)|\lambda|^n.\]

Next, for any $x \in \sigma_{\infty}$, let $y = (R - \lambda I)x$. The $n$-th component of $y$ is $y_n = x_{n-1} - \lambda x_n$ for $n \in \mathbb{N}_0$ with $x_{-1} = 0.$ Evaluating $f_\lambda$ on $y$ yields
\begin{align*}
f_{\lambda}(y) = \sum_{k=0}^{\infty} \lambda^k y_k &= \lim_{n\to\infty} \sum_{k=0}^{n} \lambda^k (x_{k-1} - \lambda x_k) \\
&= \lim_{n\to\infty} \left(-\lambda x_0 + (\lambda x_0 - \lambda^2 x_1) + \dots + (\lambda^n x_{n-1} - \lambda^{n+1} x_n)\right) \\
&= \lim_{n\to\infty} (-\lambda^{n+1} x_n).
\end{align*}
Applying $|x_n| \le (2n+1)\|x\|_{\sigma_{\infty}}$ for all $n \in \mathbb{N}_0$, we obtain
\[
|f_{\lambda}(y)| = \lim_{n\to\infty} |\lambda|^{n+1} |x_n| \le \|x\|_{\sigma_{\infty}} \lim_{n\to\infty} (2n+1)|\lambda|^{n+1} = 0.
\]
Thus, $f_{\lambda}(y) = 0$ for all $y \in \mathscr{R}(R - \lambda I)$. Continuous extension implies $f_{\lambda}(y) = 0$ on $\overline{\mathscr{R}(R - \lambda I)}$. However, for the constant sequence $e = (1, 1, 1, \dots) \in \sigma_\infty$, we have
\[
f_{\lambda}(e) = \sum_{n=0}^{\infty} \lambda^n = \frac{1}{1-\lambda} \neq 0, \qquad \lambda \in \mathbb{D}.
\]
Therefore, $f_\lambda \neq 0$, which proves that $\overline{\mathscr{R}(R - \lambda I)} \subseteq \mathscr{N}(f_\lambda) \subsetneq \sigma_{\infty}$. Hence, $\lambda \in \sigma_{r}(R, \sigma_{\infty})$.

\textbf{Case 2: $\lambda = e^{i\theta}$ with $\theta \in (0, 2\pi)$.} \\
Consider the complex sequence $y = (y_n)_{n=0}^\infty$ defined by $y_n = (n+1)\bar{\lambda}^n - n\bar{\lambda}^{n-1}$ for $n \in \mathbb{N}_0$. The partial sums of $y$ results to $\sum_{k=0}^{n} y_k = (n+1)\bar{\lambda}^n$, which implies
\[
\frac{1}{n+1}\left|\sum_{k=0}^{n} y_k\right| = |\overline{\lambda}|^n = 1 \implies \|y\|_{\sigma_\infty} = \sup_{n \ge 0} \frac{1}{n+1}\left|\sum_{k=0}^{n} y_k\right| = 1.
\]
Hence, $y \in \sigma_{\infty}$. We claim that $y \notin \overline{\mathscr{R}(R - \lambda I)}$. 

To establish this, let there exists $x \in \sigma_{\infty}$ such that $(R-\lambda I)x = y$. This gives, $x_{n} = - \frac{(n+1)(n+2)}{2} \bar{\lambda}^{n+1} + \frac{n(n+1)}{2} \bar{\lambda}^{n}$, for all $n \in \mathbb{N}_{0}$. But $\frac{1}{n+1}|\sum_{k=0}^{n}x_{k}| = \big| \frac{n+2}{2} \bar{\lambda}^{n+1} \big| = \frac{n+2}{2}$. This implies, $x \notin \sigma_{\infty}$. So, $y \notin \mathscr{R}(R- \lambda I)$.

Now we show that $\mathbb{B}(y, \frac{1}{8}) \cap \mathscr{R}(R - \lambda I) = \emptyset$, where $\mathbb{B}(y, r)$ denotes the open ball in $\sigma_\infty$ centered at $y$ with radius $r$. Suppose for contradiction that there exists $\zeta \in \mathbb{B}(y, \frac{1}{8}) \cap \mathscr{R}(R - \lambda I)$. Then $\zeta = y + \alpha = (R - \lambda I)b$ for some $b, \alpha \in \sigma_{\infty}$ with $\|\alpha\|_{\sigma_\infty} < \frac{1}{8}$. The relation $(R - \lambda I)b = y + \alpha$ leads to the explicit recurrence formula for $b_n$:
\[
b_n = -\frac{(n+1)(n+2)}{2}\bar{\lambda}^{n+1} + \frac{n(n+1)}{2}\bar{\lambda}^n + \sum_{k=0}^{n} \alpha_k \bar{\lambda}^{n+1-k}, \qquad n \in \mathbb{N}_0.
\]
Summing over $n$ yields
\[
\sum_{k=0}^{n} b_k = -\frac{(n+1)(n+2)}{2}\bar{\lambda}^{n+1} + \sum_{k=0}^{n} \sum_{t=0}^{k} \alpha_t \bar{\lambda}^{k+1-t}.
\]
Applying the reverse triangle inequality gives
\[
\frac{1}{n+1}\left|\sum_{k=0}^{n} b_k\right| \ge \frac{n+2}{2} - \frac{1}{n+1}\left|\sum_{k=0}^{n} \sum_{t=0}^{k} \alpha_t \bar{\lambda}^{k+1-t}\right|.
\]
Reordering the double summation yields
\begin{align*}
\frac{1}{n+1}\left|\sum_{k=0}^{n} \sum_{t=0}^{k} \alpha_t \bar{\lambda}^{k+1-t}\right| &= \frac{1}{n+1}\left| \bar{\lambda}\sum_{k=0}^{n}\alpha_k + \bar{\lambda}^2 \sum_{k=0}^{n-1}\alpha_k + \dots + \bar{\lambda}^{n+1}\alpha_0 \right| \\
\leq & \frac{1}{n+1}\Big|\sum_{k=0}^{n}\alpha_{k}\Big| + \frac{n}{n+1} \frac{1}{n}\Big|\sum_{k=0}^{n-1}\alpha_{k}\Big|+ \dots + \frac{1}{n+1} |\alpha_{0}| \\
			\leq & \big(1+ \frac{n}{n+1} + \dots + \frac{1}{n+1}\big) \|\alpha\|_{\sigma_{\infty}}\\
			= & \frac{n+2}{2} \|\alpha\|_{\sigma_{\infty}} < \frac{n+2}{16}.
\end{align*}
Substituting this bound back yields
\[
\frac{1}{n+1}\left|\sum_{k=0}^{n} b_k\right| > \frac{n+2}{2} - \frac{n+2}{16} = \frac{7}{16}(n+2).
\]
Taking $n \to \infty$ shows that $b \notin \sigma_\infty$, contradicting the assumption that $b \in \sigma_\infty$. Therefore, $\mathbb{B}(y, \frac{1}{8}) \cap \mathscr{R}(R - \lambda I) = \emptyset$, which implies $y \notin \overline{\mathscr{R}(R - \lambda I)}$ and proves that $e^{i\theta} \in \sigma_{r}(R, \sigma_{\infty})$ for all $\theta \in (0, 2\pi)$.

\textbf{Case 3: $\lambda = 1$.} \\
Let $y = (1000, 1000, 1000, \dots) \in \sigma_\infty$. Suppose for contradiction that $y \in \overline{\mathscr{R}(R - I)}$. Then for $\epsilon = 1$, there exists $x \in \sigma_{\infty}$ such that $\|(R - I)x - y\|_{\sigma_\infty} < 1$. By definition of the norm on $\sigma_\infty$,
\[
\frac{1}{k+1}\left|\sum_{n=0}^{k}(x_{n-1} - x_n - 1000)\right| < 1 \implies \left|\frac{x_k}{k+1} + 1000\right| < 1, \qquad \text{for all } k \in \mathbb{N}_0,
\]
where $x_{-1} = 0$. Taking the real part $\operatorname{Re}(x_k)$ yields
\[
\left|\frac{\operatorname{Re}(x_k)}{k+1} + 1000\right| < 1 \implies -1001(k+1) < \operatorname{Re}(x_k) < -999(k+1), \qquad k \in \mathbb{N}_0.
\]
Summing over $k = 0, 1, \dots, n$ gives
\[
\sum_{k=0}^{n}\operatorname{Re}(x_k) < \sum_{k=0}^{n} -999(k+1) = -999\frac{(n+1)(n+2)}{2}.
\]
Dividing by $n+1$ leads to
\[
\frac{1}{n+1}\left|\sum_{k=0}^{n}\operatorname{Re}(x_k)\right| > \frac{999(n+2)}{2} \xrightarrow{n \to \infty} \infty,
\]
which implies $x \notin \sigma_\infty$, a contradiction. Thus $y \notin \overline{\mathscr{R}(R - I)}$, proving $1 \in \sigma_r(R, \sigma_\infty)$.

Combining Cases 1–3 yields $\overline{\mathbb{D}} \subseteq \sigma_{r}(R, \sigma_{\infty}) \subseteq \sigma(R, \sigma_{\infty}) = \overline{\mathbb{D}}$, establishing $\sigma_{r}(R, \sigma_{\infty}) = \overline{\mathbb{D}}$.
\end{proof}

\begin{corollary}\label{cor:R_fine_spectrum}
The fine spectral decomposition of the unilateral shift operator $R$ on $\sigma_\infty$ consists of:
\[
\sigma_p(R, \sigma_\infty) = \emptyset, \quad \sigma_c(R, \sigma_\infty) = \emptyset, \quad \text{and} \quad \sigma_r(R, \sigma_\infty) = \sigma(R, \sigma_\infty)= \overline{\mathbb{D}}.
\]
\end{corollary}
\begin{proof}
The proof directly follows from fact that $\sigma_p(R, \sigma_\infty)$, $\sigma_c(R, \sigma_\infty)$, and $\sigma_r(R, \sigma_\infty)$ forms a disjoint partition of $\sigma(R, \sigma_\infty)$ and Theorem \ref{R_residual}.
\end{proof}

\section{Fine Spectrum of Bounded Convolution Operators on \texorpdfstring{$\sigma_\infty$}{sigma infinity}}\label{sec4}

Having established the fine spectrum of the unilateral shift operator $R$ on $\sigma_{\infty}$, we now extend our analysis to general discrete convolution operators $T_a$. To study the spectral properties of $T_a$, we must first identify the class of complex sequences $a = (a_n)_{n=0}^\infty$ for which $T_a$ defines a bounded linear operator on $\sigma_\infty$. 

Recall that for a given sequence space $X$, its multiplier space $\mathscr{M}(X)$ is defined as
\[
\mathscr{M}(X) := \Bigl\{a \in \mathbb{C}^{\mathbb{N}_{0}} : a*x \in X \text{ for all } x \in X \Bigr\}.
\]
Since $\sigma_{\infty}$ is a Banach space, an application of the Closed Graph Theorem confirms that 
\[T_{a} \in \mathscr{B}(\sigma_{\infty}) \mbox{ if and only if } a\in \mathscr{M}(\sigma_{\infty}).\]
The characterization of $\mathscr{M}(\sigma_{\infty})$ and the determination of the exact operator norm of $T_a$ were established in \cite{saha2026convolution}, as summarized in the following lemmas.

\begin{lemma} \cite[Corollary 5.9]{saha2026convolution}
The multiplier space of $\sigma_{\infty}$ is $\ell^1$, that is, $\mathscr{M}(\sigma_{\infty}) = \ell^1$. Consequently, $T_a \in \mathscr{B}(\sigma_\infty)$ if and only if $a \in \ell^1$.
\end{lemma}

\begin{lemma}
For every $a \in \ell^1$, the operator norm of $T_a$ is given by
\[
\|T_{a}\|_{\mathscr{B}(\sigma_{\infty})} = \|a\|_{\ell^1}.
\]
\end{lemma}
\begin{proof}
    The result follows directly from \cite[Theorem 5.7]{saha2026convolution} taking the weight sequence $d_n=n+1$.
\end{proof}

In this section, we analyze the spectral properties and fine spectrum classification of the convolution operator $T_a$ on $\sigma_\infty$, beginning with an algebraic inclusion relation for the spectrum. Note that if $a = (a_0, 0, 0, \ldots)$, then $T_a = a_0 I$ is a scalar operator with $\sigma(T_a) = \sigma_p(T_a) = \{a_0\}$ and $\sigma_c(T_a) = \sigma_r(T_a) = \emptyset$. To exclude this case, we assume throughout the remainder of this section that $a = (a_0, a_1, \ldots) \in \ell^1$ with $a_n \neq 0$ for at least one $n \ge 1$, ensuring $T_a$ is a well-defined, non-scalar bounded linear operator on $\sigma_\infty$.


Let $\mathfrak{B} = \mathscr{B}(\sigma_{\infty})$ denote the unital Banach algebra of bounded linear operators on $\sigma_{\infty}$ endowed with the operator norm $\|\cdot\|_{\mathfrak{B}}$, where $I \in \mathfrak{B}$ is the identity operator. On $\mathfrak{B}$, define the algebra of convolution operators
$\mathscr{M}_{\mathrm{op}} := \{T_{a} : a \in \ell^1\} \subseteq \mathfrak{B}.$ Finally, let $\mathfrak{W}$ denote the analytic Wiener algebra consisting of continuous functions $\Phi_a : \overline{\mathbb{D}} \to \mathbb{C}$ of the form
\[
\Phi_a(z) = \sum_{n=0}^{\infty} a_n z^n, \ a \in \ell^1,
\]
equipped with pointwise operations, unit element $\mathbf{1}_{\overline{\mathbb{D}}}$, and norm $\|\Phi_a\|_{\mathfrak{W}} = \|a\|_{\ell^1}$.

We are now ready to investigate the spectrum of $T_{a}$. The following inclusion relation provides a partial characterization of $\sigma(T_{a})$ using an algebraic method.
\begin{proposition}\label{prop:spectral-inclusion}
For any $a \in \ell^1$, the spectrum of $T_a \in \mathscr{B}(\sigma_{\infty})$ satisfies
\[
\sigma(T_a, \sigma_{\infty}) \subseteq \Phi_a(\overline{\mathbb{D}}).
\]
\end{proposition}

\begin{proof}
From the multiplier algebra $\mathscr{M}(\sigma_{\infty})= \ell^1$ and the commutativity relation $a*b=b*a$, the Banach algebra $\mathscr{M}_{op}$ is commutative and closed sub-algebra of $\mathfrak{B}$ containing $I$. Also, let $\mathfrak{A}:=\big(\ell^{1},*\big)$ be the Banach algebra, where the multiplication $*$ is convolution, and $e_0 = (\delta_{n0})_{n=0}^\infty$ is the unit element of $\ell^1$. By the operator norm identity $\|T_a\|_{\mathfrak{B}} = \|a\|_{\ell^1}$ the mapping
\[
\psi : \mathfrak{A} \longrightarrow \mathscr{M}_{\mathrm{op}}, \qquad \psi(a) = T_a,
\]
is an isometric algebra isomorphism satisfying $\psi(e_0) = I$, Again, the mapping
\[
\mathcal{G} : \ell^1 \longrightarrow \mathfrak{W}, \qquad \mathcal{G}(a) = \Phi_a,
\]
defines an isometric algebra isomorphism mapping $e_0$ to $\mathbf{1}_{\overline{\mathbb{D}}}$. Combining these maps yields the isometric algebra isomorphism $\theta := \mathcal{G} \circ \psi^{-1} : \mathscr{M}_{\mathrm{op}} \longrightarrow \mathfrak{W}$, as illustrated in Figure~\ref{fig:triple-isomorphism}.

\begin{figure}[htbp]
\centering
\begin{tikzcd}[column sep=1.5cm, row sep=0.8cm]
& \mathfrak{A} \arrow[dl, "\psi"', "\cong"] \arrow[dr, "\mathcal{G}", "\cong"'] & \\
\mathscr{M}_{\mathrm{op}} \arrow[rr, "\theta", "\cong"' ] & & \mathfrak{W}
\end{tikzcd}
\caption{Commutative diagram of isometric Banach algebra isomorphisms relating $\mathscr{M}_{\mathrm{op}}$, $(\ell^1, *)$, and $\mathfrak{W}$.}
\label{fig:triple-isomorphism}
\end{figure}

Consequently, all three Banach algebras $\mathscr{M}_{op}$, $\mathfrak{A}$, and $\mathfrak{W}$ are isometrically isomorphic. By Gelfand theory for commutative Banach algebras, the spectrum of an element in $\mathfrak{W}$ equals its range over the character space $\Delta(\mathfrak{W})$. Identifying $\Delta(\mathfrak{W})$ with $\overline{\mathbb{D}}$ via point evaluations $\chi_z(\Phi_a) = \Phi_a(z)$ for $z \in \overline{\mathbb{D}}$, we obtain
\[
\sigma_{\mathfrak{W}}(\Phi_a) = \{\chi_z(\Phi_a) : z \in \overline{\mathbb{D}}\} = \Phi_a(\overline{\mathbb{D}}).
\]
Invariance of the spectrum under unital algebra isomorphisms implies
\[
\sigma_{\mathscr{M}_{\mathrm{op}}}(T_a) = \sigma_{\ell^1}(a) = \sigma_{\mathfrak{W}}(\Phi_a) = \Phi_a(\overline{\mathbb{D}}).
\]
Finally, since $\mathscr{M}_{\mathrm{op}}$ is a closed, unital subalgebra of $\mathfrak{B} = \mathscr{B}(\sigma_\infty)$, standard spectral inclusion for Banach subalgebras yields
\[
\sigma_{\mathfrak{B}}(T_a) \subseteq \sigma_{\mathscr{M}_{\mathrm{op}}}(T_a) = \Phi_a(\overline{\mathbb{D}}),
\]
completing the proof.
\end{proof}

To this end, we establish a fundamental structural representation of the convolution operator $T_a$ as an power series of the right shift operator. This representation serves as the main tool for deriving the spectral properties of $T_a$.

	\begin{lemma}\label{TaPWS}
		Let $a\in \ell^{1}$ and $\Phi_{a}(z)$ denotes the power series with $a$ as sequence of coefficients. Then the convolution operator $T_a \in \mathscr{B}(\sigma_\infty)$ admits the operator-valued power series representation
\[
T_a = \Phi_a(R) = \sum_{n=0}^{\infty} a_n R^n,
\]
where the series converges absolutely to $T_a$ in the operator norm topology of $\mathscr{B}(\sigma_{\infty})$.
	\end{lemma}
	\begin{proof}
		Let $x \in \sigma_{\infty}$. Then, for any $k\in \mathbb{N}$ and $h=0,1,\dots,k-1$, we have
        \[
        \frac{1}{n+k}\Big|\sum_{i=0}^{k-h-1}x_{i}\Big| = \frac{k-h}{n+k}\frac{1}{k-h}\Big|\sum_{i=0}^{k-h-1}x_{i}\Big| \le \frac{k-h}{n+k} \|x\|_{\sigma_{\infty}} \le \|x\|_{\sigma_{\infty}},
        \]
        where $n\in \mathbb{N}_{0}$. Now, if $\Phi_{a}^{n}(z)= \sum_{k=0}^{n}a_{k}z^k$, then 
		\begin{align*}
			&\Big\| \big[T_{a} - \Phi_{a}^{n}(R)\big]\big(x\big) \Big\|_{\sigma_{\infty}} \\
			&= \Big\| \big(0, 0, \cdots , 0, \underbrace{a_{n+1} x_{0}}_{(n+1) \text{-th term}}, \underbrace{a_{n+2} x_{0} + a_{n+1} x_{1}}_{(n+2)\text{-th term}}, \cdots ,\underbrace{ \sum_{i=0}^{k-1} a_{n+k-i} x_{i}}_{(n+k) \text{-th term}}, \cdots \big)\Big\|_{\sigma_{\infty}} \\
			&= \underset{k \geq 1}{\sup} \  \frac{1}{n+k}\Big|a_{n+1}x_{0} + a_{n+2}x_{0}+a_{n+1}x_{1}+ \cdots+ \sum_{i=0}^{k-1} a_{n+k-i} x_{i}\Big| \\
			&= \underset{k \geq 1}{\sup} \frac{1}{n+k} \Big| a_{n+1} \sum_{i=0}^{k-1}x_{i}  +  a_{n+2} \sum_{i=0}^{k-2}x_{i}  + \cdots+  a_{n+k}x_{0} \Big|\\
			& \leq \underset{k \geq 1}{\sup}\ \Biggl\{\big|a_{n+1}\big|\ \frac{1}{n+k}\Bigg|\sum_{i=0}^{k-1}x_{i}\Bigg| + \big|a_{n+2}\big|\ \frac{1}{n+k}\Bigg|\sum_{i=0}^{k-2}x_{i}\Bigg|  + \\
			&  \cdots + \big|a_{n+k} \big|\ \frac{1}{n+k} \big|x_{0}\big| \Biggr\} \\
			& \leq \underset{k \geq 1}{\sup} \sum_{i=1}^{k}\big|a_{n+i}\big|\  \big\|x\big\|_{\sigma_{\infty}}\\
			& = \sum_{i=n+1}^{\infty}\big|a_{i}\big|\  \big\|x\big\|_{\sigma_{\infty}}.
		\end{align*}
		Hence, 
        \[\Big\|T_{a} - \Phi_{a}^{n}(R) \Big\|_{\mathscr{B}(\sigma_{\infty})} \leq \sum_{i=n+1}^{\infty}\big|a_{i}\big| \to 0 \mbox{ as } n \to \infty.\]
		This proves $T_{a}= \Phi_{a}(R) = \sum_{n=0}^{\infty}a_{n}R^n$.
	\end{proof}

  In the study of the spectrum and fine spectrum of $T_{a}$, we first investigate the point spectrum of $T_{a}$.  
	\begin{theorem} \label{pts-Ta-SI}
		Let $a \in \ell^1$. Then the point spectrum of the convolution operator $T_a \in \mathscr{B}(\sigma_\infty)$ is empty, i.e.,
    \[
    \sigma_p(T_a, \sigma_\infty) = \emptyset.
    \]
	\end{theorem}
	\begin{proof}
		Let $\lambda \in \mathbb{C}$ such that $\lambda \in \sigma_{p}(T_{a},\sigma_{\infty})$. Then $T_{a}(x)=\lambda x $, for some $x(\neq0)\in \sigma_{\infty}$. This gives the following system of equations:
		\begin{align*}
			a_{0}x_{0}&=\lambda x_{0}&\\
			a_{1}x_{0}+a_{0}x_{1}&=\lambda x_{1}&\\
			a_{2}x_{0}+a_{1}x_{1}+a_{0}x_{2}&=\lambda x_{2}&\\
			\hfil &\ \vdots \ &
		\end{align*}
		Consider the following cases.\\
		Case 1 $(a_{0}\neq \lambda)$: In this case $x_{n}=0, \forall n \in\mathbb{N}_{0}$. Therefore, $\lambda$ is not an eigenvalue of $T_{a}$.\\
		Case 2 $(a_{0}=\lambda)$:
		Let $k$ be the smallest index for which $x_{k}\neq 0$. Then we have
		\begin{align*}
			a_{0}x_{k}&= \lambda x_{k}\\
			a_{1}x_{k}+ a_{0}x_{k+1} &= \lambda x_{k+1}\\
			a_{2}x_{k}+ a_{1}x_{k+1}+ a_{0}x_{k+2} &= \lambda x_{k+2}\\
			& \ \ \vdots 
		\end{align*}
		This gives $a_{n} = 0$ for all $n \in \mathbb{N}$, which is not true. Therefore, $\sigma_{p}(T_{a}, \sigma_{\infty}) = \emptyset$.
	\end{proof}	

\begin{theorem} \label{RSTa}
	For any $a \in \ell^1$, the residual spectrum of the convolution operator $T_a \in \mathscr{B}(\sigma_\infty)$ is given by
\[
\sigma_r(T_a, \sigma_\infty) = \Phi_a(\overline{\mathbb{D}}).
\]
\end{theorem}

\begin{proof}
Let $\lambda \in \overline{\mathbb{D}}$. By Theorem \ref{R_residual}, $\lambda \in \sigma_{r}(R, \sigma_{\infty})$, which implies that the range $\mathscr{R}(R - \lambda I)$ is not dense in $\sigma_{\infty}$. By the Hahn--Banach Theorem, there exists a non-zero continuous linear functional $f \in \sigma_{\infty}^{\prime}$ such that $f(y) = 0$ for all $y \in \overline{\mathscr{R}(R - \lambda I)}$. 

Consequently, $f((R - \lambda I)x) = 0$ for all $x \in \sigma_{\infty}$, which gives $f(Rx) = \lambda f(x)$. By simple induction on $n$, it follows that
\[
f(R^n x) = \lambda^n f(x) \quad \text{for all } n \in \mathbb{N}_0 \text{ and } x \in \sigma_{\infty}.
\]
Recall from Theorem \ref{TaPWS} that $T_a = \Phi_a(R) = \sum_{k=0}^{\infty} a_k R^k$, where the series converges in the operator norm topology of $\mathscr{B}(\sigma_{\infty})$. Using the continuity of $f$ it follows:
\begin{align*}
			f(T_{a}(x)) = f\big(\Phi_{a}(R)x\big)= f\Bigg(\sum_{k=0}^{\infty}a_{k}R^{k}x\Bigg) &= \lim_{n\to \infty} \sum_{k=0}^{n}a_{k}f(R^{k}x)\\
			& = \lim_{n\to \infty} \sum_{k=0}^{n} a_{k}\lambda^{k}f(x)\\
			& = f(x) \lim_{n\to \infty} \sum_{k=0}^{n} a_{k}\lambda^{k}\\
			& = f(x) \Phi_{a}(\lambda)\\
			& = f \Big(\Phi_{a}(\lambda) x\Big).
		\end{align*} 
This yields $f\left( \big( T_a - \Phi_a(\lambda) I \big) x \right) = 0$ for all $x \in \sigma_{\infty}$. Since $f \neq 0$, it follows that $\overline{\mathscr{R}(T_a - \Phi_a(\lambda) I)} \neq \sigma_{\infty}$, meaning that $T_a - \Phi_a(\lambda) I$ does not have a dense range. 

Moreover, Theorem \ref{pts-Ta-SI} guarantees that $\sigma_p(T_a, \sigma_{\infty}) = \emptyset$, so $T_a - \Phi_a(\lambda) I$ is injective for every $\lambda \in \overline{\mathbb{D}}$. Therefore, $\Phi_a(\lambda) \in \sigma_r(T_a, \sigma_{\infty})$, establishing the inclusion
\[
\Phi_a(\overline{\mathbb{D}}) \subseteq \sigma_r(T_a, \sigma_{\infty}).
\]
Combining this with the spectral containment $\sigma(T_a, \sigma_{\infty}) \subseteq \Phi_a(\overline{\mathbb{D}})$ in Proposition \ref{prop:spectral-inclusion}, we conclude that
\[
\sigma_r(T_a, \sigma_{\infty}) = \Phi_a(\overline{\mathbb{D}}).
\]
\end{proof}

By combining the algebraic containment bound from Proposition~\ref{prop:spectral-inclusion} with the residual spectrum in Theorem~\ref{RSTa}, we obtain the complete spectral characterization of $T_a$.

    \begin{corollary}\label{WSTa}
    For any $a \in \ell^1$, the spectrum of the convolution operator $T_a \in \mathscr{B}(\sigma_\infty)$ is given by
\[
\sigma(T_a, \sigma_\infty) = \Phi_a(\overline{\mathbb{D}}).
\]
    \end{corollary}
    \begin{proof}
        From Proposition \ref{prop:spectral-inclusion}, we have $\sigma(T_{a}, \sigma_{\infty})\subseteq \Phi_{a}(\overline{\mathbb{D}})$, while Theorem \ref{RSTa} gives, $\sigma_{r}(T_{a}, \sigma_{\infty})= \Phi_{a}(\overline{\mathbb{D}})$. Combining these two results, we have $\Phi_{a}(\overline{\mathbb{D}})=\sigma_{r}(T_{a}, \sigma_{\infty}) \subseteq \sigma(T_{a}, \sigma_{\infty})\subseteq \Phi_{a}(\overline{\mathbb{D}})$, which implies our result.
    \end{proof}

Finally, applying the classical fine spectrum partition $\sigma(T_a, \sigma_\infty) = \sigma_p(T_a) \cup \sigma_r(T_a) \cup \sigma_c(T_a)$, we immediately deduce the continuous spectrum of $T_a$ is empty.
    
	\begin{corollary}
	    For any $a \in \ell^1$, the continuous spectrum of the convolution operator $T_a \in \mathscr{B}(\sigma_\infty)$ is empty:
\[
\sigma_c(T_a, \sigma_\infty) = \emptyset.
\]
	\end{corollary}

	\begin{example}
	Consider the sequence $a = (a_n)_{n=0}^\infty \in \ell^1$ defined by $a_0 = 0$ and $a_n = \frac{1}{(n-1)!}$ for $n \ge 1$, that is,
\[
a = \left(0, 1, 1, \frac{1}{2!}, \frac{1}{3!}, \dots\right).
\]
Since $\|a\|_{\ell^1} = \sum_{k=0}^\infty \frac{1}{k!} = e < \infty$, the convolution operator $T_a \in \mathscr{B}(\sigma_\infty)$ is bounded. The associated symbol function $\Phi_a \in \mathfrak{W}$ is given by
\[
\Phi_a(z) = \sum_{n=0}^\infty a_n z^n = z \sum_{k=0}^\infty \frac{z^k}{k!} = ze^z, \qquad z \in \overline{\mathbb{D}}.
\]
Applying Theorem~\ref{RSTa} yields that the point spectrum $\sigma_p(T_a, \sigma_\infty)$ and the continuous spectrum $\sigma_c(T_a, \sigma_\infty)$ are empty and the spectrum consists entirely of the residual spectrum:
\[
\sigma(T_a, \sigma_\infty) = \sigma_r(T_a, \sigma_\infty) = \Phi_a(\overline{\mathbb{D}}).
\]
The geometric region $\Phi_a(\overline{\mathbb{D}})$, representing the spectrum of $T_a$ in the complex plane, is depicted in the Figure \ref{fig:img1}.
		\begin{figure}[h] 
			\centering
			\includegraphics[width=.36\linewidth]{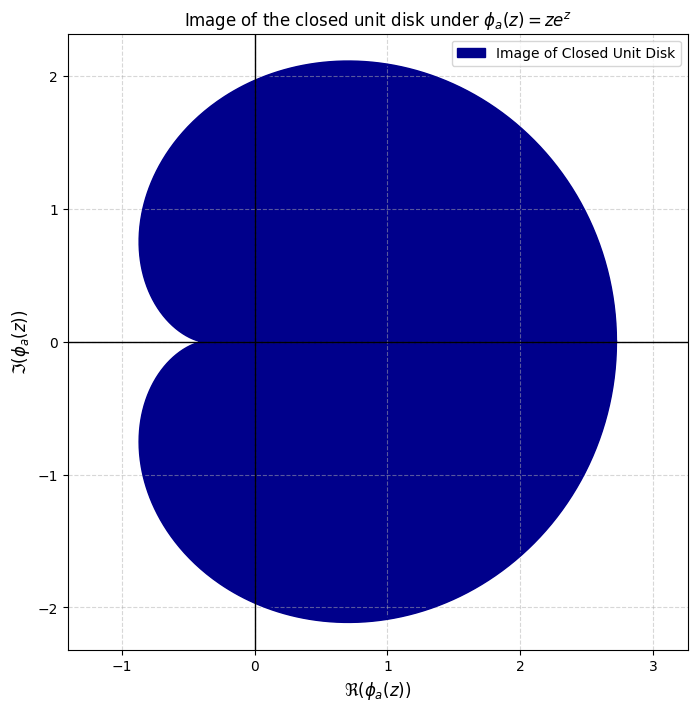}
			\caption{Image of closed unit disk under $\Phi_a(z)$}
			\label{fig:img1}
		\end{figure}
	\end{example}

	


\section{Fine Spectrum of Operator Polynomials and Band Matrices}\label{sec5}

Let $X \supset c_{00}$ be an arbitrary complex Banach space, $A \in \mathscr{B}(X)$ be a bounded linear operator, and
\[
p(z) = \sum_{k=0}^{n} a_k z^k \in \mathbb{C}[z]
\]
be a complex polynomial of degree $n \ge 1$. Naturally $p(A) = \sum_{k=0}^{n} a_k A^k$ is a bounded linear operator on $X$. 

The primary objective of this section is to develop a unified framework that derives the complete fine spectrum partition of $p(A)$ directly from the corresponding fine spectral components of $A$. 

A paramount application of this framework arises in the study of finite band triangular Toeplitz matrices. Setting $A = R$, the right shift operator on a sequence space, the operator polynomial $p(R)$ yields the discrete convolution operator
\[
T_a = p(R) = \sum_{k=0}^{n} a_k R^k,
\]
generated by the finite sequence $a = (a_0, a_1, \dots, a_n, 0, 0, \dots)$. The matrix representation of $T_a$ corresponds to a lower triangular Toeplitz matrix with finitely many non-zero bands. In a similar manner, by choosing $A=L$, the left shift operator, one obtains an upper triangular banded Toeplitz matrix. Although the spectral properties of banded triangular matrices on classical sequence spaces have been studied extensively in the literature (see, e.g., \cite{altun2011fine,birbonshi2017some,furkan2007fine,karakaya2010fine} and references therein), existing approaches rely on the separability of the underlying space. Consequently, non-separable sequence spaces, most notably $\ell^{\infty}$, are less studied. Specifically, the upper triangular case
has remained outside the scope of unified fine spectral characterizations.

By contrast, the fine spectral mapping principles established herein require no separability assumptions on $X$. This yields a powerful, coordinate-free methodology to determine the exact fine spectrum of lower or upper triangular finite band matrices across both separable and non-separable sequence spaces.

For this purpose, we begin by establishing the precise spectral mapping behavior for each fine component, starting with a known result on the point spectrum.

	\begin{lemma} \cite[Theorem 10.33, p.266]{rudin1991functional} \label{Rudin_pts}
		Let $A \in \mathscr{B}(X),\ \Omega$ open in $\mathbb{C},\ \sigma(A,X)\subset \Omega$. If $f$ is analytic on $\Omega$ and $f$ is not constant in any component of $\Omega$, then $f(\sigma_{p}(A,X))= \sigma_{p}(f(A),X)$.
	\end{lemma}
	
	Next we establish relationships between $\sigma_{r}(p(A),X)$ and $p\Big(\sigma_{r}(A,X)\Big)$, as well as between $\sigma_{c}(p(A),X)$ and $p\Big(\sigma_{c}(A,X)\Big)$, where $A\in \mathscr{B}(X), \ p(z)= \sum_{k=0}^{n}a_{k}z^{k}$. In this regard, we have the following theorem.

\begin{lemma} \label{RSMT}
Let $X$ be a complex Banach space, $A \in \mathscr{B}(X)$, and $p(z) \in \mathbb{C}[z]$ be a non-constant polynomial. Then the residual spectrum of $p(A)$ is given by
\[
\sigma_r(p(A), X) = p(\sigma_r(A, X)) \setminus \sigma_p(p(A), X).
\]
\end{lemma}
\begin{proof}
Let $\lambda \in \sigma_r(A, X)$. By definition, the range $\mathscr{R}(A - \lambda I)$ is not dense in $X$, i.e., $\overline{\mathscr{R}(A - \lambda I)} \neq X$. By the Hahn--Banach Theorem, there exists a non-zero continuous linear functional $f \in X^{\prime}$ such that $f(y) = 0$ for all $y \in \overline{\mathscr{R}(A - \lambda I)}$. 

This implies that $f((A - \lambda I)x) = 0$ for all $x \in X$, which gives $f(Ax) = \lambda f(x)$. By a straightforward induction argument, we obtain
\[
f(A^k x) = \lambda^k f(x) \quad \text{for all } k \in \mathbb{N}_0 \text{ and } x \in X.
\]
Writing $p(z) = \sum_{k=0}^n a_k z^k$, the linearity of $f$ yields
\[
f(p(A)x) = f\left( \sum_{k=0}^n a_k A^k x \right) = \sum_{k=0}^n a_k f(A^k x) = \sum_{k=0}^n a_k \lambda^k f(x) = p(\lambda) f(x) = f(p(\lambda) x)
\]
for all $x \in X$. Consequently, $f\left( \big( p(A) - p(\lambda)I \big)x \right) = 0$ for all $x \in X$. Since $f \neq 0$, it follows that the range $\mathscr{R}\big( p(A) - p(\lambda)I \big)$ is not dense in $X$. 

Therefore, $p(\lambda) \in \sigma_r(p(A), X) \cup \sigma_p(p(A), X)$. Since $\sigma_r(p(A), X) \cap \sigma_p(p(A), X) = \emptyset$, if $p(\lambda) \notin \sigma_p(p(A), X)$, then $p(\lambda) \in \sigma_r(p(A), X)$. This establishes the inclusion:
\[
p(\sigma_r(A, X)) \setminus \sigma_p(p(A), X) \subseteq \sigma_r(p(A), X).
\]

To prove the reverse inclusion, let $\mu \in \sigma_r(p(A), X)$. By the Fundamental Theorem of Algebra, we can factor the polynomial $p(z) - \mu$ as
\[
p(z) - \mu = c(z - z_1)(z - z_2) \cdots (z - z_n),
\]
where $c \in \mathbb{C} \setminus \{0\}$ and $z_1, z_2, \dots, z_n \in \mathbb{C}$. Consequently,
\[
p(A) - \mu I = c(A - z_1 I)(A - z_2 I) \cdots (A - z_n I).
\]
Since $\mu \in \sigma_r(p(A), X)$, the operator $p(A) - \mu I$ is injective, which implies that each commuting factor $(A - z_k I)$ must be injective for $k = 1, 2, \dots, n$. Furthermore, $\mu \in \sigma_r(p(A), X)$ implies that $\overline{\mathscr{R}(p(A) - \mu I)} \neq X$. 

If $\overline{\mathscr{R}(A - z_k I)} = X$ held for all $k = 1, 2, \dots, n$, the composition of these operators with dense ranges would also yield a dense range for $p(A) - \mu I$, a contradiction. Thus, there exists at least one $k \in \{1, 2, \dots, n\}$ such that $\overline{\mathscr{R}(A - z_k I)} \neq X$. Since $(A - z_k I)$ is injective, $z_k \in \sigma_r(A, X)$. 

Because $p(z_k) - \mu = 0$, we have $\mu = p(z_k) \in p(\sigma_r(A, X))$. Finally, since $\sigma_r(p(A), X)$ is disjoint from $\sigma_p(p(A), X)$, we conclude that $\mu \in p(\sigma_r(A, X)) \setminus \sigma_p(p(A), X)$, proving that
\[
\sigma_r(p(A), X) \subseteq p(\sigma_r(A, X)) \setminus \sigma_p(p(A), X).
\]
Combining both inclusions completes the proof.
\end{proof}

\begin{theorem}\label{Fine spectrum of operator poly}
Let $X$ be a complex Banach space, $A \in \mathscr{B}(X)$, and $p(z) \in \mathbb{C}[z]$ be a non-constant polynomial. Then the fine spectral subdivisions of $p(A)$ satisfy:
\begin{enumerate}
    \item[\rm (i)] $\sigma(p(A), X) = p(\sigma(A, X))$;
    \item[\rm (ii)] $\sigma_p(p(A), X) = p(\sigma_p(A, X))$;
    \item[\rm (iii)] $\sigma_r(p(A), X) = p(\sigma_r(A, X)) \setminus p(\sigma_p(A, X))$;
    \item[\rm (iv)] $\sigma_c(p(A), X) = p(\sigma(A, X)) \setminus p(\sigma_p(A, X) \cup \sigma_r(A, X))$.
\end{enumerate}
\end{theorem}
\begin{proof}
\noindent\textbf{(i)} This is the classical Spectral Mapping Theorem for bounded linear operators.

\smallskip
\noindent\textbf{(ii)} Directly follows form Lemma \ref{Rudin_pts}.

\smallskip
\noindent\textbf{(iii)} By Lemma \ref{RSMT}, the residual spectrum satisfies $\sigma_r(p(A), X) = p(\sigma_r(A, X)) \setminus \sigma_p(p(A), X)$. Substituting the identity $\sigma_p(p(A), X) = p(\sigma_p(A, X))$ from part (ii) directly yields
\[
\sigma_r(p(A), X) = p(\sigma_r(A, X)) \setminus p(\sigma_p(A, X)).
\]

\smallskip
\noindent\textbf{(iv)} Since the point, residual, and continuous spectra form a disjoint partition of the spectrum $\sigma(p(A), X)$, we have
\[
\sigma_c(p(A), X) = \sigma(p(A), X) \setminus \Big( \sigma_p(p(A), X) \cup \sigma_r(p(A), X) \Big).
\]
Using parts (ii) and (iii), the union of the point and residual spectra simplifies to
\begin{align*}
\sigma_p(p(A), X) \cup \sigma_r(p(A), X) &= p(\sigma_p(A, X)) \cup \Big( p(\sigma_r(A, X)) \setminus p(\sigma_p(A, X)) \Big) \\
&= p(\sigma_p(A, X)) \cup p(\sigma_r(A, X)) \\
&= p(\sigma_p(A, X) \cup \sigma_r(A, X)).
\end{align*}
Combining this with $\sigma(p(A), X) = p(\sigma(A, X))$ from part (i) establishes
\[
\sigma_c(p(A), X) = p(\sigma(A, X)) \setminus p(\sigma_p(A, X) \cup \sigma_r(A, X)),
\]
which completes the proof.
\end{proof}

As an immediate application of the general fine spectral mapping principles established in Theorem~\ref{Fine spectrum of operator poly}, we now specialize to the classes of $n$-band lower and upper triangular Toeplitz operators. Let $X(\supset c_{00})$ be a complex Banach sequence space on which both the right shift operator $R$ and the left shift operator $L$ are bounded. Given a polynomial $p(z) = \sum_{k=0}^{n} a_k z^k$ associated with a finitely supported sequence $a = (a_0, a_1, \dots, a_n, 0, 0, \dots) \in c_{00}$, any lower and upper triangular $n$-band operator can be represented as operator polynomials $p(R) = \sum_{k=0}^{n} a_k R^k$ and $p(L) = \sum_{k=0}^{n} a_k L^k$, respectively. Consequently, the entire fine spectral decompositions of such banded matrices are completely inherited from that of $R$ or $L$.


First, we apply Theorem \ref{Fine spectrum of operator poly} to settle the fine spectrum of the lower triangular band matrix $p(R)$ on the non-separable sequence space $\ell^\infty$, providing a complete characterization that has been missing in the literature.

For this purpose, recall that the spectral components of the right shift operator $R$ on $\ell^\infty$, as established in \cite{helein2014spectral}, are given by
\[
\sigma(R,\ell^\infty)
=\sigma_r(R,\ell^\infty)
=\overline{\mathbb{D}},
\qquad
\sigma_p(R,\ell^\infty)
=\sigma_c(R,\ell^\infty)
=\emptyset.
\]

We also establish the fine spectral properties of the left shift operator $L$ on the non-separable space $\ell^\infty$.

\begin{lemma}\label{lem:fine-spectrum-L}
Let $L \in \mathscr{B}(\ell^\infty)$ be the left shift operator given by $L(x_0, x_1, x_2, \dots) = (x_1, x_2, x_3, \dots)$. Then $\|L\|_{\mathscr{B}(\ell^\infty)} = 1$, and its spectral decomposition is
\[
\sigma(L, \ell^\infty) = \overline{\mathbb{D}}, \qquad \sigma_p(L, \ell^\infty) = \overline{\mathbb{D}}, \qquad \sigma_r(L, \ell^\infty) = \emptyset, \qquad \sigma_c(L, \ell^\infty) = \emptyset.
\]
\end{lemma}

\begin{proof}
It is straightforward that $\|L\|_{\mathscr{B}(\ell^\infty)} = 1$, implying $\sigma(L, \ell^\infty) \subseteq \overline{\mathbb{D}}$. For any $\lambda \in \overline{\mathbb{D}}$, setting $x = (1, \lambda, \lambda^2, \dots) \in \ell^\infty$ yields $Lx = \lambda x$. Hence every $\lambda \in \overline{\mathbb{D}}$ is an eigenvalue of $L$, giving $\overline{\mathbb{D}} = \sigma_p(L, \ell^\infty) \subseteq \sigma(L, \ell^\infty) \subseteq \overline{\mathbb{D}}$. Disjointness of the fine spectral components completes the proof.
\end{proof}

Applying Theorem~\ref{Fine spectrum of operator poly} alongside the known spectral data for $R$ and $L$ on $\ell^\infty$ we obtain the fine spectrum of both lower and upper triangular banded operators acting on $\ell^\infty$.

\begin{theorem}\label{thm:banded-linfty}
Let $p(z) = \sum_{k=0}^n a_k z^k$ be a non-constant complex polynomial. The fine spectral breakdowns of the lower banded operator $p(R)$ and the upper banded operator $p(L)$ on $\ell^\infty$ are given by:
\begin{enumerate}
    \item[(i)] {Lower Triangular Band Operator $p(R)$:}
    \[
    \sigma(p(R), \ell^\infty) = p\big(\overline{\mathbb{D}}\big), \quad \sigma_p(p(R), \ell^\infty) = \emptyset, \quad \sigma_r(p(R), \ell^\infty) = p\big(\overline{\mathbb{D}}\big), \quad \sigma_c(p(R), \ell^\infty) = \emptyset.
    \]
    \item[(ii)] {Upper Triangular Band Operator $p(L)$:}
    \[
    \sigma(p(L), \ell^\infty) = p\big(\overline{\mathbb{D}}\big), \quad \sigma_p(p(L), \ell^\infty) = p\big(\overline{\mathbb{D}}\big), \quad \sigma_r(p(L), \ell^\infty) = \emptyset, \quad \sigma_c(p(L), \ell^\infty) = \emptyset.
    \]
\end{enumerate}
\end{theorem}
\begin{remark}
    The spectrum and point spectrum in part (i) of Theorem \ref{thm:banded-linfty} were previously obtained in \cite{birbonshi2017some}, however the residual and continuous spectrum were not determined. The result in part (ii) was previously established in \cite{altun2011fine}.
\end{remark}

To illustrate the consistency and computational power of Theorem \ref{Fine spectrum of operator poly}, we show how it effortlessly reproduces known results on classical separable spaces without requiring laborious space-specific matrix derivations.

\begin{example}\label{ex:SpFpc}
Consider the operator polynomials $p(R)$ and $p(L)$ acting on the sequence space $c$. From \cite{birbonshi2017some}, the spectral components for $R$ and $L$ on $c$ are:
\[
\begin{aligned}
\sigma(R, c) &= \overline{\mathbb{D}}, & \sigma_p(R, c) &= \emptyset, & \sigma_r(R, c) &= \mathbb{D} \cup \{1\}, & \sigma_c(R, c) &= \partial\mathbb{D} \setminus \{1\}; \\
\sigma(L, c) &= \overline{\mathbb{D}}, & \sigma_p(L, c) &= \mathbb{D} \cup \{1\}, & \sigma_r(L, c) &= \emptyset, & \sigma_c(L, c) &= \partial\mathbb{D} \setminus \{1\}.
\end{aligned}
\]
Direct application of Theorem~\ref{Fine spectrum of operator poly} yields:
\begin{enumerate}
    \item[(i)] {Fine Spectrum of $p(R)$ on $c$:}
{\small    \[
    \sigma(p(R), c) = p\big(\overline{\mathbb{D}}\big), \quad \sigma_p(p(R), c) = \emptyset, \quad \sigma_r(p(R), c) = p(\mathbb{D}) \cup \{p(1)\}, \quad \sigma_c(p(R), c) = p\big(\overline{\mathbb{D}}\big) \setminus \Big[p(\mathbb{D}) \cup \{p(1)\}\Big].
    \]}
    \item[(ii)] {Fine Spectrum of $p(L)$ on $c$:}
   {\small \[
    \sigma(p(L), c) = p\big(\overline{\mathbb{D}}\big), \quad \sigma_p(p(L), c) = p(\mathbb{D}) \cup \{p(1)\}, \quad \sigma_r(p(L), c) = \emptyset, \quad \sigma_c(p(L), c) = p\big(\overline{\mathbb{D}}\big) \setminus \Big[p(\mathbb{D}) \cup \{p(1)\}\Big].
    \]}
\end{enumerate}
These results fully match the corresponding results in \cite{birbonshi2017some}, which further supports our findings.
\end{example}

	

\section{Conclusion}
In this paper, we have obtained a complete fine spectral characterization of general convolution operators $T_a$ acting on the non-separable space $\sigma_\infty.$ Furthermore, we have established a general fine spectral mapping framework for operator polynomials acting on arbitrary complex Banach spaces without any separability assumptions (Theorem \ref{Fine spectrum of operator poly}). This approach unifies and streamlines the existing literature on triangular banded Toeplitz operators across sequence spaces, without separability assumption. Crucially, by bypassing separability, our work yields the full fine spectral characterization of banded Toeplitz operators on non-separable sequence spaces such as $\ell^\infty$, resolving a longstanding gap in the literature.


\section*{Declarations}

\begin{itemize}
	\item Availability of data and materials: Not applicable.
	\item Competing interests: The authors declare that they have no competing interests.
	\item Funding: Not applicable.
	\item Authors' contributions: All authors contributed equally to this work.
    \item Acknowledgement: Mr. Sayan Saha sincerely acknowledge the University Grants Commission (UGC), Govt. of India, for financial support in the form of fellowship (NTA Ref. No.: 221610024673).
	\end{itemize}

\bibliographystyle{abbrv}
\bibliography{Bibliography_All}	
\end{document}